\documentclass[11pt]{article} 
\usepackage{a4,amssymb, amsmath, amsthm}
\usepackage{color}

\newtheorem{thm}{Theorem}[section]
\newtheorem{cor}[thm]{Corollary}

\newtheorem{lem}[thm]{Lemma}
\newtheorem{pro}[thm]{Proposition}

\theoremstyle{remark}
\newtheorem{rem}[thm]{Remark}

\newcommand{\Rd}{{\mathbb{R}^d}}

\begin{document}

\title{A deterministic optimal design problem for the heat equation\\\vskip 0.8cm}
\author{ Heiko Gimperlein\thanks{Maxwell Institute for Mathematical Sciences and Department of Mathematics, Heriot--Watt University, Edinburgh, EH14 4AS, United Kingdom, email: h.gimperlein@hw.ac.uk}  \thanks{Institute for Mathematics, University of Paderborn, Warburger Str.~100, 33098 Paderborn, Germany} \and Alden Waters\thanks{Department of Mathematics, University College London, Gower Street, London, WC1E 6BT, United Kingdom, email: alden.waters@ucl.ac.uk }}
\date{}

\maketitle \vskip 0.5cm
\begin{abstract}
\noindent For the heat equation on a bounded subdomain $\Omega$ of $\Rd$, we investigate the optimal shape and location of the observation domain in observability inequalities. A new decomposition of $L^2(\Rd)$ into heat packets allows us to remove the randomization procedure and assumptions on the geometry  of $\Omega$ in previous works. The explicit nature of the heat packets gives new information about the observability constant in the inverse problem. 
\end{abstract}

\vskip 1.0cm

\section{Introduction}\label{intro}

This article considers an optimal design problem for the heat equation: What is the optimal shape and location of a thermometer if we would like to reconstruct the heat distribution in a domain? We aim to introduce techniques from microlocal analysis, related to heat packet decompositions, in order to address a rigorous formulation of this question. The explicit nature of the heat packets sheds light on certain randomization assumptions and technical hypotheses in previous works. \\

Optimal design problems for the placement of sensors have attracted much interest in analysis and computational mathematics. Recent works include \cite{blr, burq, burqgerard, hh, ptz}, which consider observability and optimal design problems for  wave and Schr\"{o}dinger equations. For the heat equation, an observability estimate goes back to \cite{lr}, and \cite{ptzheat} considers a simplified optimal design problem with random initial conditions. The dissipativity of the heat equation makes the reconstruction of high frequencies much less stable, one of the issues addressed in this article.\\

For a precise statement of a model problem, we consider a solution $u(t,x)$ to the heat equation 
$u_t = \Delta u$
in a smooth domain $\Omega \subset \Rd$
with homogeneous Dirichlet boundary conditions and arbitrary initial condition $u(0,\cdot)\in C^\infty_c(\Omega)$, c.f Appendix Theorem \ref{wp} for well-posedness estimates. Given $T>0$ and a measurable subset $\omega \subset \Omega$, we denote by $C_T(\omega)$  the best constant such that 
\begin{equation}\label{myeq}
C_T(\omega) \int_\Omega |u(T,x)|^2 \ dx \leq \int_0^T \int_\omega |u(t,x)|^2\ dt\ dx 
\end{equation}
when $u(0, \cdot)\in C_c^\infty(\Omega)$. $C_T(\omega)$ gives an account for the well-posedness of the inverse problem of reconstructing $u$ from measurements over $[0,T]\times \omega$.\\ 

This article studies for which subdomains the heat equation is observable, and whether there are optimal ones.  In the paper \cite{kv}, a parametrix to the linear Schrodinger equation is built. We build solutions to the heat equation after their model, but we change the variable $it\mapsto t$. We rescale their original high frequency data to obtain information about $L^2(\Omega)$ data. Thus, a key new ingredient in our investigation is a decomposition of the initial data $$u = \sum_{n} c_n \phi_n\ , \quad \phi_n(t,x)=\Big(\frac{\sigma}{\sqrt{2\pi}(\sigma^2 +t)}\Big)^{d/2}\exp\Big(i x \cdot \xi_n - t|\xi_n|^2 - \frac{|x+2i\xi_nt|^2}{4 (\sigma^2 +t)}\Big)$$ of $u$ into heat packets following \cite{kv}. Here $c_n$ are constants,  $\xi_n$ belongs to a lattice in $\Rd$, and $\sigma$ is a frequency parameter; we specify them later.  While this decomposition is valid in whole space, we may reduce to this case after approximating the heat kernel on $\Omega$. The heat packets replace the propagated Dirichlet eigenfunctions of $\Omega$, but still allow almost explicit calculations.

We exhibit a family of optimal design problems which accurately approximate the true optimal design problem when the initial data is \emph{not} high frequency and the time scale is of the same size as the frequency. We let $C_T^A(\omega)$ denote the constant associated to this approximate optimal design problem.  We let $0<\epsilon\ll 1$ be a small constant; the precise conditions on $\epsilon$ will be specified in terms of a fixed positive constant $\eta$ with $\eta<1$ and the timescale $T$. The parameter $\epsilon$ describes the number of initial data points $\xi_n$ we need in our frame-based approximation.  The parameter $\eta$ describes the convergence of the approximate observability constant $C_T^A(\omega)$ to $C_T(\omega)$. In particular we show that $\epsilon$ needs to increase  with the timescale $T$, and $\epsilon$ may increase monotonously with $\eta$. 

For computational clarity, we assume that $\psi\in C_c^{\infty}(\Omega)$, with supp$\psi\subset[-1,1]^d$. Our initial data $u(0,x)$ we say is of the form
\begin{align*}
u(0,x)=\epsilon_{0}^{-\frac{d}{2}}\psi\left(\frac{x}{\epsilon_{0}}\right).
\end{align*}
with $\epsilon_{0}$ a number in $(0,\frac{1}{2}\textrm{diam}(\Omega))$ which is sufficiently small enough for $\mathrm{supp}u(0,x)\subset \Omega$. The data is at minimal distance $\delta$ from $\partial\Omega$ with $\delta\in (\sqrt{\epsilon_0},1)$. We search for an optimal subdomain $\omega$ in the set
\begin{align*}
&\mathcal{M}_{M}=\{\omega \subset \Omega\ |  \ 
 \omega \,\, \textrm{is measureable and of Lebesgue measure}\,\, |\omega|=M|\Omega| \}\ .
\end{align*}
It accounts for the fact that we measure the solutions on a subdomain of $\Omega$ with a fixed volume.  

The classical approach of \cite{ptzheat} involves
separation of variables using a basis of eigenfunctions $\Delta \Psi_j = - \lambda_j \Psi_j$. Here one would decompose the solution into this basis as
$u(t,x) = \sum_{j=1}^\infty a_j e^{-t\lambda_j} \Psi_j(x).$
If we define $b_j = a_j e^{-T\lambda_j}$, the question becomes to  examine 
\begin{align*}
C_T(\omega) &=\inf_{\sum_{j=1}^\infty |b_j|^2 = 1} \int_0^T \int_\omega \left|\sum_{j=1}^\infty b_j e^{\lambda_jt} \Psi_j(x)\right|^2 dx\ dt\\
& = \inf \sigma \left(\frac{e^{(\lambda_j + \lambda_k)T}-1}{\lambda_j + \lambda_k} \ \int_\omega \Psi_j(x) \Psi_k(x)\ dx\right) \ ,
\end{align*}
where $\sigma$ denotes the spectrum of the matrix. This is a hard spectral problem since little is known about $\int_\omega \Psi_j(x) \Psi_k(x)\ dx$ even in the case of the disk: the restriction of inner products of arbitrary Bessel functions to subsets $\omega \subset \Omega$ cannot be computed explicitly.

 In order to avoid this problem, Privat, Trelat and Zuazua \cite{ptzheat} replace $a_j$ by a sequence of real-valued random variables $\{\beta_j^\nu a_j\}_{j \in \mathbb{N}, \nu \in \mathcal{X}}$ and thereby introduce a random field $u_\nu(t,x)$. The $\beta_j^\nu$ are independent identically distributed, of mean $0$, variance $1$ and fast decay (e.g.~Bernoulli). They then study the case of an averaged observability constant
\begin{align*}
C_T^{rand}(\omega) & = \inf \sigma \left(\frac{e^{(\lambda_j + \lambda_k)T}-1}{\lambda_j + \lambda_k} \ (\mathbb{E} \beta_j^\nu \beta_k^\nu)\ \int_\omega \Psi_j(x) \Psi_k(x)\ dx\right) \\
 &= \inf \sigma \left(\frac{e^{2\lambda_j T}-1}{2\lambda_j} \ \int_\omega \Psi_j(x)^2\ dx\right)\\
&= \inf_{j \in \mathbb{N}}  \frac{e^{2\lambda_j T}-1}{2\lambda_j} \ \int_\omega \Psi_j(x)^2\ dx\ .
\end{align*}
From the above we see that removing the randomization hypothesis reduces to a difficult spectral problem: To determine inner products on the subdomain $\omega$ of eigenfunctions for the Dirichlet Laplacian on $\Omega$.  We instead introduce an approximate problem, denoted by $A$ which consists of a finite number of inner products of explicit functions. For the uniqueness results, we must randomize the initial data. We also assume that $u(t,x)$ intersects the observation domain $\omega$ at some time $t_0\in (0,T)$ and this intersection is bounded below by $m_0$. We set $m_1=m_0\|\psi\|_{C^0(\Omega)}^{-1}\eta$. This criterion is related to \cite{blr}. 

Our main result is: 
\begin{thm}\label{maintheorem}
Let $M \in (0,1)$ and $T<\min\{\epsilon_0^{2+s}m_1(C_{s,d}\|\psi(x)\|_{C^{2+s}(\Omega)})^{-1}, m_1\delta^4\}$, $s>d/2$ an integer, and $\Omega \subset \Rd$, then \\
a) $C_T(\omega)\gtrsim C^{A}_T(\omega)>0$ for all $\omega \in \mathcal{M}_M$.\\ 
b) There exists a unique $\omega^\ast \in \mathcal{M}_M$ such that $C^{A,rand}_T(\omega^\ast)\geq C^{A,rand}_T(\omega)$ for all  $\omega \in \mathcal{M}_M$. \\
c) The subset $\omega^\ast$ is open and semi--analytic. In particular, the Hausdorff measure of $\partial \omega^\ast$ is $0$. 
\end{thm}
The randomized initial condition required in \cite{ptzheat}, which also included strong implicit assumptions on the geometry of $\Omega$. In order to study the randomized problem further, Privat, Trelat and Zuazua require strong assumptions on the level sets of these eigenfunctions ( ($H_1$) in \cite{ptzheat}) and thereby on $\Omega$. Here, we show that removing randomization amounts to the computation of a finite number of inner products of explicit functions, for which the assumptions on level sets can be verified. We show that the problem can be modeled for arbitrary initial data in $L^2(\Omega)$ by Gaussians, suggesting that a deeper analysis of the heat kernel could be useful in including the randomized terms for the uniqueness results. In particular we prove one cannot ignore the non-randomized terms entirely in this scenario. We provide quantitative upper and lower bounds on their contribution to the observability constant. 

The number of grid points required for the reconstruction of $C_T^A(\omega)$ increases drastically at high frequencies. Of course, the initial data must be in $H_0^1(\Omega)$ for the solution to even exist, so in some sense this phenomenon is to be expected. The poor convergence for the high frequencies suggests that the use of monochromatic waves, which is introduced by randomization introduces a significant loss of information in the optimal design problem. In particular, the randomized observability constant only provides an upper bound $C_T^{rand}(\omega)\geq C_T(\omega)$, while we have
\begin{align*}
C_T(\omega)\geq C_T^A(\omega).
\end{align*}
In essence, we are looking at the worst case scenarios for the behavior of $\omega$. The reason we cannot extend to infinite times, is that the errors in the approximation argument dominate for large times. We also give an example of high frequency data showing the difficulty of removing the randomization assumption of \cite{ptz}. \\

It should be noted that \cite{ptzheat} is the first paper to construct an optimal observability domain $\omega$ in \emph{any} sense, and we rely on their ideas \emph{extensively}. It is an earlier result of \cite{lr} that $C_T(\omega)>0$ for compact connected smooth manifolds.\\

The heat packets lead to a numerical approximation scheme, where one determines a subset $\omega_N^\ast \in \mathcal{M}_M$ with the best constant in \eqref{myeq} among finite linear combinations of heat packets 
\begin{align*}
u =\sum_{|n|\leq N} c_n \phi_n\ .
\end{align*}
 A basic problem is the convergence of $\omega_N^\ast$ to $\omega^\ast$ as $N \to \infty$ in Hausdorff distance.
There are well-known counterexamples due to spillover phenomena \cite{hh} for hyperbolic equations. However, the penalization of high frequencies in the heat equation allows to prove such a result, similar to \cite{ptzheat} for randomized initial conditions:

\begin{thm} \label{nospilltheorem}
a) $\omega_N^\ast$ is uniquely determined, semi--analytic and open.\\
b) The sequence $\omega_N^\ast$ converges to $\omega^\ast$ as $N \to \infty$. In fact, $\omega_N^\ast = \omega^\ast$ for large $N$. \\ 
\end{thm}
Section \ref{model1} considers the auxiliary problem in the half-space, in which explicit computations based on the Gaussian heat packets are possible. The model for $C_T^A(\omega)$ and the proof of Theorem \ref{maintheorem} a) is given Section \ref{model1}. The proofs of Theorems \ref{maintheorem} b),c) and \ref{nospilltheorem} are established together in Section \ref{proof}. The construction of the frame implies that we have precision over $N$. For us, the results indicate $N$ increases with an increase in $T$. \\

\noindent  \emph{Notation:} For $z \in \mathbb{C}^d$ we write $\|z\|^2 = \overline{z}^T z$ for the square of its length and $|z|^2 = z^T z$ for the analytic extension of the absolute value on $\mathbb{R}^d$. 

\section{Truncation to model problem in upper half plane}\label{model1}

We consider our initial data to be a smooth compactly supported function, and explain how this can be used to obtain information on optimal design problems for more general $L^2(\Rd)$ data later in the paper. For general conditions, the evolution is approximated by a superposition of heat packets. We start with the upper half plane as our model and use the Feynman-Kac principle to move the problem to the bounded domain $\Omega$.

We assume diam$(\Omega)\gg\eta$, and $|\Omega|=1$. We extend $\psi$ by $0$ to a smooth function on $\Rd$ and use the following lemma to decompose the function into heat packets. For the computation using Kac's principle it is convenient if the origin, without loss of generality, has the maximal distance $\frac{1}{2}$diam$(\Omega)$ to $\partial\Omega$. 

We decompose $\psi$ as follows: 
\begin{lem}\label{kvlemma}[Generalization of Lemma 4.14 in \cite{kv}]
Fix $\psi \in C_c^\infty(\Rd)$ as above and $\eta\in (0,1)$ and $\epsilon_0\in (0,\frac{1}{2}\mathrm{diam}(\Omega))$. We define $\psi_{\epsilon_0}(x)=\epsilon_0^{-d/2}\psi(x/\epsilon_0)$. Let $\delta$ denote the distance of the center of $\psi_{\epsilon_0}(x)$ to the origin. Then there is a small $0<\epsilon\ll 1$ depending on $\eta,\epsilon_0$ such that the following holds: For $\sigma = \sqrt{\epsilon_0\delta} \log(\epsilon^{-1})$, $L = \sigma \log \log(\epsilon^{-1})$ there are coefficients $\{c^n\}_{n \in \mathbb{Z}^d}$ and $\xi_n = \frac{n}{L}$ with: 
\begin{align}\label{frame1}
\left\|\psi_{\epsilon_0}(x)- \sum_{n \in \mathbb{Z}^d} \frac{c^n}{(2 \pi \sigma^2)^{d/4} }\exp\left(- \frac{|x|^2}{4 \sigma^2} + i \xi_n\cdot x \right)\right\|_{L^2(\omega)} \leq\eta \|\psi_{\epsilon_0}(x)\|_{L^2(\omega)}\ .
\end{align}
Moreover, 
\begin{align*}
|c^n| \lesssim_{\psi,k} C_{k, \psi} (\sigma \epsilon_0 L^{-2})^{d/2} \min\left\{1, (\epsilon_0 |\xi_n|)^{-k}\right\}
\end{align*}
for all $k \in \mathbb{N}$. One may choose $c^n= 0$ unless 
\begin{align*}
n \in \mathcal{S} = {\textstyle\left\{ k \in \mathbb{Z}^d : \frac{1}{\epsilon_0 \log \log(\epsilon^{-1})} \leq |\xi_k|\leq \frac{\log \log(\epsilon^{-1})}{\epsilon_0}\right\}}.
\end{align*}
if we include all the vectors then $\eta=0$. 
\end{lem}
The proof in \cite{kv} is for $d=3$. We include it here to show that the smallness condition on $\epsilon$ is compatible with later restrictions on $\epsilon$ in our paper. The proof in \cite{kv} is for high frequency data only, which we rescaled to obtain our result, since $\epsilon_0$ is arbitrary. The rescaling affects $\sigma$ and $L$. This implies the set of required vectors $\mathcal{S}$ becomes \emph{very} large as $\epsilon_0\rightarrow 0$. The integer $k$ reflects the regularity of $\psi$, and it is possible to bound $k$ by $d$ provided $\epsilon$ is small, and still produce an $\ell^2$ sequence $c_n$. The construction of the $\ell^2$ bound requires $\epsilon_0<\delta\leq 1$, but it should still be possible to construct an $\ell^2$ sequence as long as $\delta>0$.  
\begin{proof}
 We define
\begin{align*}
c_n=(2\pi\sigma)^{-d}\int\limits_{[-\pi L,\pi L]^d}\psi_{\epsilon_0}(x)(2\pi\sigma^2)^{d/4}\exp\left(\frac{|x|^2}{4\sigma^2}-in\cdot \frac{x}{L}\right)\,dx
\end{align*}
and
\begin{align*}
\phi_n(x)=(2\pi\sigma^2)^{-d/4}\exp\left(-\frac{|x|^2}{4\sigma^2}+in\cdot\frac{x}{L}\right)\ .
\end{align*} 
Note that $\|\phi_n(x)\|_{L^2(\mathbb{R}^d)}=1$. As the Fourier series of a smooth function converges uniformly, we have by Plancherel's theorem:
\begin{align}\label{fconv}
\psi_{\epsilon_0}(x)=\sum\limits_{n\in \mathbb{Z}^d}c_n\phi_n(x)
\end{align}
on $[-\pi L,\pi L]^d$.
If $\epsilon$ is sufficiently small, $\mathrm{supp}\psi_{\epsilon_0}\subset [-\frac{\pi L}{2},\frac{\pi L}{2}]^d$. The series \eqref{fconv} gives the result immediately if $\omega\subset \Omega$, which it is. The error comes from the truncation, which we will prove a bound on using the estimates on the $|c_n|$. 

We now prove an upper bound on the coefficients. Let $M_1$ be a constant such that $\|\psi\|_{C^d(\Omega)}\leq M_1$. From the definition of $c_n$ we obtain
\begin{align*}
|c_n|\leq \frac{\sigma^{d/2}}{L^d}\textstyle{\|\psi_{\epsilon_0}(x)\exp\left(\frac{|x|^2}{4\sigma^2}\right)}\|_{L^1(\Rd)}\lesssim \frac{(\sigma\epsilon_0)^{\frac{d}{2}}}{L^d}\ .
\end{align*}
In order to derive the upper bound, we use integration by parts. Let $\mathbb{D}=i\frac{Ln}{|n|^2}\cdot\nabla$. We see that $\mathbb{D}^k\exp(-in\cdot\frac{x}{L})=\exp(-in\cdot\frac{x}{L})$. The adjoint of $\mathbb{D}$ is given by $\mathbb{D}^t=-i\nabla\cdot \frac{Ln}{|n|^2}$. We therefore obtain the following
\begin{align*}
|c_n|&=(2\pi L)^{-d}\left|\int\limits_{\Rd}\mathbb{D}^k\exp\left(in\cdot \frac{x}{L}\right)\psi_{\epsilon_0}(x)(2\pi\sigma^2)^{d/4}\exp\left(\frac{|x|^2}{2\sigma^2}\right)\,dx\right| \\&
=(2\pi L)^{-d}\left|\int\limits_{\Rd}\exp\left(-in\cdot\frac{x}{L}\right)(\mathbb{D}^t)^k\left[\epsilon_0^{-d/2}\psi_{\epsilon_0}(x)(2\pi\sigma^2)^{d/4}\exp\left(\frac{|x|^2}{2\sigma^2}\right) \right]\,dx\right| 
\\& \lesssim L^{-d}\left(\frac{L}{|n|}\right)^k\left(\frac{\sigma}{\epsilon_0}\right)^{d/2}\sum\limits_{|\alpha|\leq k}\|\partial^{\alpha}\left[\psi\left(\frac{x}{\epsilon_0}\right)\exp\left(\frac{|x|^2}{4\sigma^2}\right)\right]\|_{L^1(\Rd)}
\\&\lesssim_{k}
\frac{(\epsilon_0\sigma)^{d/2}}{L^d}\left(\frac{L}{\epsilon_0|n|}\right)^kM_1 \ .
\end{align*}
The claimed upper bound follows. 
For the remaining claim we notice that 
\begin{align*}
\gamma_{nm}=\int\limits_{\mathbb{R}^d}\phi_n(x)\overline{\phi_m(x)}\,dx=\exp\left(-\frac{\sigma^2}{2L^2}|n-m|^2\right)\ . 
\end{align*}
Fix $N\in \mathbb{N}$. If $n\leq N$, we use the upper bound for $c_n$ to estimate
\begin{align*}
\|\sum\limits_{|n|\leq N}c_n\phi_n(x)\|^2_{L^2(\Rd)}\lesssim \frac{(\sigma\epsilon_0)^d}{L^{2d}}\sum\limits_{|n|,|m|\leq N}\exp\left(-\frac{\sigma^2}{2L^2}|n-m|^2\right)\lesssim \left(\frac{\epsilon_0N}{L}\right)^dM_1\ .
\end{align*}
For $n\geq N$, we use the second upper bound for $c_n$ with $k=d$ to obtain 
\begin{align*}
\left|\sum\limits_{|n|\geq N}c_n\phi_n\right|_{L^2(\Rd)}&\lesssim\frac{(\sigma\epsilon_0)^d}{L^{2d}}\left(\frac{L}{\epsilon_0}\right)^{2d}\sum\limits_{|n|\geq|m|\geq N}\frac{1}{|n|^3}\frac{1}{|m|^3}\exp\left(-\frac{\sigma^2}{2L^2}|n-m|^2\right)\\&\lesssim\left(\frac{\sigma}{\epsilon_0}\right)^d\sum\limits_{|n|\geq|m|\geq N}\frac{1}{|m|^{2d}}\lesssim\left(\frac{L}{\epsilon_0N}\right)^dM_1 \ .
\end{align*}
We conclude
\begin{align*}
\|\sum\limits_{|n|\leq \frac{L}{\epsilon_0\log\log(\epsilon^{-1})}}c_n\phi_n\|_{L^2(\Rd)}^2+\|\sum\limits_{|n|\geq \frac{L}{\epsilon_0}\log\log(\epsilon^{-1})}c_n\phi_n\|_{L^2(\Rd)}^2\lesssim(\log\log(\epsilon^{-1}))^{-d}M_1\ .
\end{align*} 
The assertion follows provided that
\begin{align}\label{ob1}
(\log\log(\epsilon^{-1}))^{-d}M_1\ll\eta\ .
\end{align}
\end{proof}

This observation of \cite{kv} gives a precise expansion of the initial condition into Gaussian heat packets 
\begin{align*}
\phi_{n,x_0} (x) = \frac{1}{(2 \pi \sigma^2)^{d/4} }\exp\left(- \frac{|x-x_0|^2}{4 \sigma^2} + i \xi_n\cdot (x-x_0) \right)\ ,
\end{align*}
which evolve in $\Rd$ according to simple analytic formulas and are centered around a suitable $x_0$. In the decomposition we selected, $x_0=0$,  by our choice of origin. Here $x_0$ denotes the fact that we could have had more that one $\psi_{\epsilon_0}$ in our decomposition for $u(0,x)$. We leave $x_0$ in to show that it is difficult to do the Gramian computations when the centers are not suitably separated.

For simplicity we write \begin{align*}
\phi_{n,x_0}(t,x) = [\exp(t\Delta_\Rd)\phi_{n,x_0}](t,x)\ ,
\end{align*}
 where we have that 
\begin{align*}
\phi_{n,x_0}(t,x) = \left(\frac{\sigma}{\sqrt{2 \pi}(\sigma^2 + t)}\right)^{d/2} \exp\left(- \frac{|x-x_0+2 i \xi_n t|^2}{4 (\sigma^2+t)} - t|\xi_n|^2 + i \xi_n\cdot (x-x_0)  \right) \ .
\end{align*}
We let
\begin{align*}
k^{\mathbb{R}^d}(t,x,y)=(4\pi t)^{-d/2}\exp(-|x-y|^2/4t)
\end{align*}
denote the heat kernel in $\Rd$ and $k^{\Omega}$ the corresponding heat kernel in $\Omega$. We take Kac's principle, Proposition 6.3.1 in \cite{ard}, to approximate the evolution of the heat packets in the upper half plane. Kac's principle states:
\begin{pro}
Assume that $\Omega\subset \Rd$ is smooth, and let $x\in\Omega$ be fixed. For every $y\in \Omega$ we let 
\begin{align*}
t_0(y)=\frac{d(y,\Gamma)^2}{2d}.
\end{align*}
It follows that
\begin{align}\label{model}
0\leq k^{\mathbb{R}^d}(t,x,y)-k_{\Omega}(t,x,y)\leq \left\{
\begin{array}{lr}
(4\pi t)^{-\frac{d}{2}}\exp(-d(y,\Gamma)^2/4t) \quad \mathrm{if} \qquad  t\leq t_0(y),\\ \\
(4\pi t_0(y))^{-\frac{d}{2}}\exp(-d/2) \,\quad \mathrm{if} \,\qquad t>t_0(y).  \\
\end{array}
\right.
\end{align}
\end{pro}
A short computation using Kac's principle shows that: 
\begin{lem}\label{fk}
We have for $0<t<\eta_0\delta^4$, and $\|\psi\|_{C^0(\Omega)}<M_0$. 
\begin{align*}
\|\exp(t\Delta_{\Omega})\psi_{\epsilon_0}(x)-\exp(t\Delta_{\Rd})\psi_{\epsilon_0}(x)\|_{C([0,T]\times\Omega)}\leq \frac{\eta_0}{2}M_0
\end{align*}
whenever $\epsilon_0<\delta^2<1$. 
\end{lem}
\begin{proof} 
We would like an estimate on 
\begin{align}\label{kernel}
\sup_{t,x}|\int\limits_{\Rd}(k^{\mathbb{R}^d}(t,x,y)-k_{\Omega}(t,x,y))\psi_{\epsilon_0}(y)\,dy\ |\ .
\end{align}
We divide $\Omega$ into a region I, defined by $d(y,\partial\Omega)>t^{1/4}$, and its complement, region II. Using Kac's principle we can bound \eqref{kernel} from above by 
\begin{align*}
& \sup_{t,x}\int\limits_{I}(4\pi t)^{-d/2}\exp\left(-\frac{d(y,\partial\Omega)^2}{4t}\right)|\psi_{\epsilon_0}(y)|\,dy\\& \quad +\sup_{t,x}\int\limits_{II}\left(\frac{2\pi}{d}\right)^{-d/2}\left(d(y,\partial\Omega)\right)^{-d}\exp(-d/2)|\psi_{\epsilon_0}(y)|\,dy\ .
\end{align*} 
 We introduce coordinates near the boundary $\partial\Omega$ with $y_1=d(y,\partial\Omega)$ and $y'=(y_2,\dots, y_d)$. 
 For the integral in region I, we have after change of variables 
\begin{align*}
\int\limits_{t^{1/4}/\epsilon_0}^{\mathrm{diam}(\Omega)/\epsilon_0}\int\limits_{y'/\epsilon_0}\epsilon_0^d(4\pi t)^{-d/2}\exp\left(\frac{-y_1^2\epsilon_0^2}{4t}\right)|\psi(y_1,y')|\,dy'\,dy_1\leq{\epsilon_0^{\frac{d}{2}}\frac{\exp\left(-\frac{1}{4t^{1/2}}\right)}{(4\pi t)^{\frac{d}{2}}}} M_0 |\partial \Omega| \ .
\end{align*} 
In region II the integral can be estimated by 
\begin{align*}
&\int\limits_{\delta}^{t^{1/4}}\left(\frac{2\pi}{d}\right)^{-d/2}\exp(-d/2)y_1^{-d}\int_{y'}|\psi_{\epsilon_0}(y_1,y')|\, dy'\,dy_1\\ &=\epsilon_0^{\frac{d}{2}}\left(\frac{2\pi}{d}\right)^{-d/2}\exp(-d/2)\int\limits_{\delta/\epsilon_0}^{t^{1/4}/\epsilon_0}y_1^{-d }\int_{y'/\epsilon_0}|\psi(y_1,y')|\, dy'\,dy_1\\&=\epsilon_0^{\frac{d}{2}}\left(\frac{2\pi}{d}\right)^{-d/2}\exp(-d/2)  \frac{M_0 \epsilon_0^{d-1}}{d-1} |\partial \Omega|\ (\delta^{1-d} - t^{\frac{1-d}4})\ . 
\end{align*}
and the integral is empty if $t^{1/4}\leq \delta$. If we impose this condition, then quick inspection of the second integral gives the desired result.  
\end{proof}
Initial data closer to the boundary would result in the construction of a reflected heat packet, which we leave as an area of improvement in our analysis.  For the construction of the parametrix, we impose the condition $|\omega|M_0\eta_0<\eta\|u(t_0,x)\|_{L^2(\omega)}$ for some $t_0\in (0,T)$. Because we are dividing by the measure of the set $|\omega|$, for $\eta_0$ to exist, we impose the condition similar to \cite{blr}: 
\\
($\bullet$) The support of the solution $u(t,x)$ must intersect the observation set $\omega$ at a point $x_0$ for at least one time $t_0$, and $|u(t_0,x_0)|>m_0$.
\\
\\
The condition implies $\eta_0<\frac{m_0}{M_0}\eta$, as $m(\omega)M_0\eta_0<\eta\|u(0,x)\|_{L^2(\omega)}$.

\begin{lem}[Analogue of Proposition 4.7 in \cite{kv}] 
The evolution near the boundary is given by
\begin{align*}
\exp(t\Delta_{\Omega})\psi_{\epsilon_0}(x)=\sum\limits_n c_n\phi_{n,x_0}(t,x)+v(t,x)\ ,
\end{align*}
where
\begin{align*}
\|v(t,x)\|_{C^0([0,T];L^2(\omega))}\leq \eta\|\psi_{\epsilon_0}(x)\|_{L^2(\omega)}\ .
\end{align*}
\end{lem}
\begin{proof}
This estimate requires two steps. For the first step we have:
\begin{align*}
\|\exp(t\Delta_{\Rd})\psi_{\epsilon_0}-\exp(t\Delta_{\Omega})\psi_{\epsilon_0}\|_{L^2(\omega)}\leq \frac{\eta}{2}\|\psi_{\epsilon_0}(x)\|_{L^2(\omega)}
\end{align*}
by Kac's principle and ($\bullet$). For the second step by inequality \eqref{frame1} and the parabolic maximal principle 
\begin{align*}
&\left\|\exp(t\Delta_{\Rd})\psi_{\epsilon_0}-\sum\limits_nc_n\phi_{n,x_0}(t,x)\right\|_{L^2(\omega)}\\& = \left\|\exp(t\Delta_{\Rd})\left(\psi_{\epsilon_0}-\sum\limits_nc_n\phi_{n,x_0}(x)\right)\right\|_{L^2(\omega)}\\&\leq \|\exp(t\Delta_{\Rd})\|_{L^2(\omega) \to L^2(\omega)}\|\psi_{\epsilon_0}-\sum\limits_nc_n\phi_{n,x_0}(x)\|_{L^2(\omega)} \\& \leq \frac{\eta}{2}\|\psi_{\epsilon_0}(x)\|_{L^2(\omega)}\ ,
\end{align*}
c.f.~\cite{davies}, Theorem 1.3.3. 
\end{proof}
In particular, we notice that away from this regime we can approximate the evolution of the heat packets by the evolution in $\mathbb{R}^d$. This approximation simplifies the analysis considerably, so we choose to focus on this regime as our model case.   

The key ingredient will be precise estimates for the inner products of heat packets on $\Rd$, which follow from the explicit Gaussian shape of $\phi_{n,x_0}$. Analogously to \cite{ptz}, we consider the Gramian matrix $G$ corresponding to the evolved heat packets:
$$G_{nm}(x_0,y_0,\omega) = \int_0^T \int_\omega \phi_{n,x_0}(t,x) \overline{\phi_{m,y_0}(t,x)} \ dx \ dt\ .$$
Inspired by techniques used for Gaussian frames in \cite{waters}, we would like to show that the largest contribution to the observability constant comes from the diagonal terms, and that the largest contributions are due to low--frequency heat packets. In \cite{ptzheat} a randomization assumption was required to remove the off-diagonal terms. 
   
 To simplify the notation, set $$A(t,n) = \exp\left(-\frac{2 t\sigma^2|\xi_n|^2}{\sigma^2 + t}\right).$$ We define $B(x_0,\sqrt{\sigma^2+t})$ as 
\begin{align}\label{truncation}
B(x_0,\sqrt{\sigma^2+t})=\{x: |x-x_0|\leq 2(t+\sigma^2)^{1/2} \}\ .
\end{align}

\begin{pro}\label{omegabound} Let $\epsilon_1<1$ and $T=\epsilon_1\sigma^2$. We assume $\epsilon$ in Lemma \ref{kvlemma} is small enough such that $\Omega\subset [-\sigma^2,\sigma^2]^d$. We let $C_d$ be a fixed constant depending on the dimension only. We have the following estimates:\\

a) Lower bound for the diagonal terms: 
\begin{align}\label{upper}
& G_{nn}(x_0, x_0, \omega) \geq C_d(\exp(-1))\sigma^{d}\int_0^T \frac{|\omega \cap B(x_0, \sqrt{\sigma^2+t})|}{|B(x_0, \sqrt{\sigma^2+t})|}A(t,n) \ dt \ . 
\end{align}
b) Upper bound for the diagonal terms: 
\begin{align}
& G_{nn}(x_0, x_0, \omega) \leq C_d\sigma^{d}(1+\mathrm{erfc}(1))(\mathrm{erf}(1))^{-1}\int_0^T \frac{|\omega \cap B(x_0, \sqrt{\sigma^2+t})|}{|B(x_0, \sqrt{\sigma^2+t})|} A(t,n) \ dt \ 
\end{align}
c) An upper bound for the off-diagonal terms:
\begin{align} C_d|\omega|\int_0^T\left(\frac{\sigma^2}{\sigma^2+t}\right)^{d/2}A^{1/2}(t,n)A^{1/2}(t,m)\exp\Big(- \frac{(x_0-y_0)^2}{4(\sigma^2 + t)}\Big)\,dt\
\end{align}
d) An equality for the off--diagonal terms when $|n|,|m|<L$, $\sigma^2+t<1$, $\omega$ is radially symmetric with $\mathrm{supp}\omega\subset [-R,R]^d$:
\begin{align}\label{loss}
&|G_{nm}(x_0, x_0, \omega)| \\  \nonumber
&= C_d|\omega|\mathrm{erf}{R}\int_0^T\left(\frac{\sigma^2}{\sigma^2+t}\right)^{d/2}\left(1+\left({o(1)}{(\sigma^2+t)}\right)^{\frac{d}{2}}\right)(A^{1/2}(t,n)A^{1/2}(t,m))\,dt \ .
\end{align} 
\end{pro}
\begin{proof}
The estimates follow from the explicit formula: 
\begin{align}\label{innerprod}
&\int_0^T \int_\Rd \phi_{n,x_0}(t,x) \overline{\phi_{m,x_0}(t,x)} \ dx \ dt=\\& \int_0^T \int_\Rd \exp\left(-\frac{x^2}{2(\sigma^2+t)}+ix\cdot\left((\xi_n-\xi_m)\frac{\sigma^2}{\sigma^2+t}\right)\right) \exp\left(-(\xi_n^2+\xi_m^2)\frac{t\sigma^2}{\sigma^2+t}\right)\ dx \ dt \nonumber
\end{align}

In a) we integrate the Gaussian over the intersection of $\omega$ over the standard--deviation part of the integrand around $x_0$. We recall the following definitions of the error and complementary error functions:
\begin{align*}
\frac{2}{\sqrt{\pi}}\int\limits_b^{\infty}\exp(-x^2)\,dx=\mathrm{erfc}(b) \ , \qquad \frac{2}{\sqrt{\pi}}\int\limits_0^{b}\exp(-x^2)\,dx=\mathrm{erf}(b).
\end{align*}
The error function has the following asymptotic behavior \cite{as}:
\begin{align}\label{asymptotic}
\mathrm{erfc}(b)\sim \frac{1}{(1+a_1+a_2b^2+a_3b^3+a_4b^4)^4} \qquad \forall b\geq 0
\end{align}
with $a_1=0.278393, a_2 = 0.230389, a_3=0.000972, a_4=0.078108$. 

For part $a$), we start by integrating
\begin{align*}
G_{nn}(x_0,x_0,\mathbb{R}^d)
\end{align*}
in $x$, and notice that
\begin{align*}
\frac{2}{\sqrt{\pi}}\int\limits_{B}\frac{\exp\left(-\frac{x^2}{\sqrt{\sigma^2+t}}\right)}{\sqrt{\sigma^2+t}}\simeq\mathrm{erf}(1) \ .
\end{align*}
Then we may replace the limits of integration in the inner product over $\mathbb{R}^d$ to those of the ball $B(x_0, \sqrt{\sigma^2+t})$, with small error as dictated by (\ref{asymptotic}). The desired result follows by a change of variables in $x$ and using
\begin{align*}
\inf\limits_B \exp(-x^2)=\exp(-1);
\end{align*}
here the limit is over the rescaled ball. The result in part $b$) follows similarly. For the result in part $c$) we see 
\begin{align}
&\left|\int_0^T \int_\Rd \phi_{n,x_0}(t,x) \overline{\phi_{m,y_0}(t,x)} \ dx \ dt\right|  \leq \int_0^T \left(\frac{\sigma^2}{\sigma^2 + t}\right)^{d/2} \times \cdots \nonumber\\   & \quad \times \exp\Big(\frac{ t^2(|\xi_n|^2+|\xi_m|^2)}{\sigma^2 + t} - \frac{(x_0-y_0)^2}{8(\sigma^2 + t)}-K(\xi_n - \xi_m)^2 \ - \dots \nonumber \\& \quad  -t(|\xi_n|^2+|\xi_m|^2\Big)\,dt\ . \label{innerprodupper}
\end{align}
We notice that $\exp(-K(\xi_n-\xi_m)^2)$ is decaying for short times. Unfortunately, we cannot easily exploit the decay in $K$. In particular, when integrating with respect to $x$ in \eqref{innerprod}, this amounts to an estimate on 
\begin{align}\label{fourierK}
\left|\int\limits_{\omega}\exp(ik\cdot x)\exp(-ax^2)\,dx\right|
\end{align}
for some $k,a$ constants without losing information on the oscillatory part. By Hardy's uncertainly principle, c.f.~\cite{taohardy} Theorem 1, the only way a Gaussian will Fourier transform to another Gaussian is if the initial data is Gaussian, and we are integrating over all of $\Rd$. Otherwise, we obtain a larger rate of decay. We define a smooth cut-off function $\chi_{\omega}$, which is 1 on the support of $\omega$. We see that 
\begin{align*}
&|\langle \chi_\omega,\exp(-ax^2)\exp(ikx)\rangle_{L^2(\Rd)}|\leq \\& \|\chi_\omega\|_{L^2(B_a)}^2 \|\exp(-ax^2)\exp(ikx)\|^2_{L^2(\Rd)}\leq C_d m(\omega)a^{-d/2}
\end{align*}
with $C$ independent of $a$ and $k$, while
\begin{align*}
\left(\frac{\pi}{a}\right)^{\frac{d}{2}}\exp\left(-\frac{\pi^2k^2}{a}\right)=\int\limits_{\Rd}\exp(-ax^2)\exp(ikx)\,dx.
\end{align*}
Thus, we can bound $|G_{nm}(x_0,x_0,\omega)|$ with a loss of exponential decay. It follows that \eqref{innerprodupper} is bounded by 
\begin{align*}
C_dm(\omega)\int_0^T\left(\frac{\sigma^2}{\sigma^2+t}\right)^{d/2}\exp\left(-(\xi_n^2+\xi_m^2)\frac{t\sigma^2}{\sigma^2+t}\right)\exp\Big(- \frac{(x_0-y_0)^2}{4(\sigma^2 + t)}\Big)\,dt\ .
\end{align*}
Recall as a modification of Lemma 3 in \cite{taohardy}: 
\begin{lem}\label{fbound}
Suppose $|f(x)|\leq \exp(-a x^2)$ for all $x$ in $\{x: |x|\leq R\}$ and some $a>0$. Then $\hat{f}(\xi)$ is smooth and one has the bound for $\xi\in I$:
\begin{align}
\hat f(\xi) =\frac{C_d|R|\mathrm{erf}{R}}{\sqrt{a}}\left(1+o(1)\frac{|\xi|^2}{a}\right)
\end{align}
for any interval $I$, such that $|I|<1$. 
\end{lem}
\begin{proof}
The Fourier transform has the Taylor series expansion
\begin{align}
\hat{f}(\xi)=\hat{f}(0)+\xi\partial_{\xi}\hat{f}(0)+\frac{\xi^2}{2}\partial_{\xi}\hat{f}(0)+O(|\xi|^3)\ .
\end{align}
We can then calculate
\begin{align}
\partial_{\xi}\hat{f}(\xi)|_{\xi=0}=\int\limits_{|x|<R}(2\pi i x)^k\exp(-a x^2)\,dx
\end{align}
\end{proof}
Our function which is the integrand in \eqref{innerprod} satisfies the criterion of Lemma \ref{fbound} in $1d$.  We apply the Lemma which generalizes easily to $\Rd$. It follows from the equality \eqref{loss} the off-diagonal terms contribute substantially to the matrix norm of a Gramian if $\sigma^2<1$ and $|n|,|m|\leq L$. Examining the leading order term part of the contribution in \eqref{loss}, and the upper bound in c) we see
\begin{align*}
&\mathrm{erf}{R}\inf_t((A^{-1/2}(t,n)A^{1/2}(t,m)))\|\phi_{n,x_0}(t,x)\|_{L^2((0,T)\times\Rd)}^2\leq\\& C_d\int_0^T\left(\frac{\sigma^2}{\sigma^2+t}\right)^{d/2}(A^{1/2}(t,n)A^{1/2}(t,m))\,dt\leq \\& \sup_t((A^{-1/2}(t,n)A^{1/2}(t,m))) \|\phi_{n,x_0}(t,x)\|_{L^2((0,T)\times\Rd)}^2 
\end{align*}
Both prefactors $\inf\limits_t((A^{-1/2}(t,n)A^{1/2}(t,m)))$ and $\sup\limits_t((A^{-1/2}(t,n)A^{1/2}(t,m)))$ can contribute substantially when summing over $n\neq m$. 
\end{proof}

We need the following short time estimate. 
\begin{lem}\label{bootstrap}
Assume that the $C^{2+s}$ norm of $\psi$ is bounded by $M_2$ with $s>d/2$ an integer. Then for all  $0\leq t<C_{s,d}^{-1}M_2^{-1}\eta_0\epsilon_0^{2+s}$  we have 
\begin{align*}
|u(t,x)-\psi_{\epsilon_0}(x)|\leq \eta_0 .
\end{align*}
with $C_{s,d}$ a constant depending on $d$ and $s$ only. 
\end{lem}
\begin{proof}
The mean value theorem states that $\exists t_0\in (0,1)$ such that
\begin{align}\label{MVT}
|u(t,x)-\psi_{\epsilon_0}(x)|\leq |\partial_tu(t_0,x)|t\ .
\end{align} 
We know by definition $\partial_tu=\Delta u$, and by Sobolev embedding with $s>d/2$
\begin{align*}
\|\Delta u \|_{C^0(\Omega)}\leq C_{s,d}\|u(t,x)\|_{H^{2+s}(\Omega)}\leq C_{s,d}\|u(0,x)\|_{H^{2+s}(\Omega)}
\end{align*}
The last line follows from standard energy estimates (Appendix, Theorem \ref{wp}) and the finite expansion from Lemma \ref{kvlemma}.  Our choice of timescale clearly works by scaling.
\end{proof}

We conclude from the precise values of the upper and lower bounds in Proposition \ref{omegabound} that: 
\begin{cor}\label{majorization}
Let $T<\varepsilon$ and $\varepsilon$ sufficiently small. We have the following bound 
\begin{align}\label{sandwich}
\frac{\eta}{2}\frac{\int_0^T \int_\Omega \chi_\omega(x) |u(t,x)|^2 dx\ dt}{\int_\Omega |u(T,x)|^2 dx}&\leq \nonumber \frac{\|\sum\limits_nc_n\phi_{n,x_0}(t,x)\|_{L^2(\omega\times (0,T))}^2}{\|\sum\limits_nc_n\phi_{n,x_0}(T,x)\|^2_{L^2(\Omega)}}\\&\leq \frac{2}{\eta}\frac{\int_0^T \int_\Omega \chi_\omega(x) |u(t,x)|^2 dx\ dt}{\int_\Omega |u(T,x)|^2 dx}\ ,
\end{align}
with $C^{A}_T(\omega)\lesssim C_T(\omega)$, and $C^{A}_T(\omega)>0$. The replacement is possible even if we consider a finite number $N$ of $\phi_{n,x_0}(t,x)$ with $|n|\leq N$ as dictated by Lemma \ref{kvlemma}. 
\end{cor}
\begin{proof}[Proof of Theorem \ref{maintheorem} a) and Corrollary \ref{majorization}]. 
We recall that $\|\phi_{n.x_0}(0,x)\|^2_{L^2(\Rd)}=1$. Recall that $\eta$ is the small frame parameter given to us by \eqref{frame1}.  We make the natural assumption which we call the bootstrap assumption, 
\begin{align*}
|u(t,x)-u(0,x)|<\eta_0  \Rightarrow \|u(t,x)\|_{L^2(\omega)}\geq 2\eta \|u(0,x)\|_{L^2(\omega)}\ ,
\end{align*}
which is valid for short times and small $\eta_0$ such that $|\omega|\eta_0<\eta\|u(t_0,x)\|_{L^2(\omega)}$ for some $t_0\in (0,T)$. Because we are dividing by the measure of the set $|\omega|$, for $\eta_0$ to exist, we must have the criterion ($\bullet$) hold. The question of validity of this assumption for arbitrary $\epsilon_0$ is given by Lemma \ref{bootstrap}, which fails as $\epsilon_0\rightarrow 0$ because the timescale shrinks with $\epsilon_0$. We also know
\begin{align*}
\|u(t,x)\|_{L^2(\omega)}&\leq  \|\exp(t\Delta_{\Omega}) (u(0,x)-\sum\limits_nc_n\phi_{n,x_0}(0,x) )\|_{L^2(\omega)}+\|\sum\limits_nc_n\phi_{n,x_0}(t,x)\|_{L^2(\omega)}\\&\leq  \eta\|u(0,x)\|_{L^2(\omega)}+\|\sum\limits_nc_n\phi_{n,x_0}(t,x)\|_{L^2(\omega)}\ .
\end{align*}
By the bootstrap assumption and the energy estimates in the Appendix, we have
\begin{align*}
\|u(t,x)\|_{L^2(\omega)}-\eta\|u(0,x)\|_{L^2(\omega)}\geq \eta\|u(0,x)\|_{L^2(\omega)}\geq  \eta\|u(t,x)\|_{L^2(\omega)}.
\end{align*}
We see also that
\begin{align*}
\|\sum\limits_nc_n\phi_{n,x_0}(T,x)\|_{L^2(\Omega)}\leq \|u(T,x)\|_{L^2(\Omega)}+\eta\|u(0,x)\|_{L^2(\Omega)}\leq 2\|u(T,x)\|_{L^2(\Omega)}\ ,
\end{align*}
again by the bootstrap assumption and the energy estimates. The desired result \eqref{sandwich} follows. We know from \cite{lr} that the left hand side of \eqref{sandwich} is nonzero, so $C_T^A(\omega)$ is nonzero. 

Therefore we have made only the following assumptions on $\epsilon$ in terms of the parameter $\eta$:
\begin{align*}
&0<(\log\log(\epsilon^{-1}))^{-d}M_1\ll\eta \ .
\end{align*}
This was used in Lemma \ref{kvlemma} on the condition \eqref{frame1} precisely dictated by \eqref{ob1}. We also made the short time assumption in Lemma \ref{bootstrap} and condition ($\bullet$):
\begin{align*}
t<C_{s,d}^{-1}\eta_0\epsilon_0^{2+s}M_2^{-1}, \qquad |\omega|\eta_0m_0<M_0\eta\|u(t_0,x)\|_{L^2(\omega)}\,\, \textrm{for some}\quad t_0\in (0,T)
\end{align*}
and the assumption $\Omega\subset[-\sigma^2,\sigma^2]^d$ for the bounds a,b) in Prop \ref{omegabound} to hold. 
We notice that the short timescale is only necessary for the bounds on little $\omega$. For longer timescales, the parametrix itself is valid on $\Omega$. The timescale for the parametrix to be valid is largely dictated by Kac's principle and can be improved by including reflections at $\partial \Omega$. 
\end{proof}

\begin{rem} For an arbitrary collection of data that there is no interaction, e.g.~$G_{nn}(x_0,y_0)$ can be not that much different from $G_{nn}(x_0,x_0)$ unless there are some strong hypotheses on the separation of $x_0,y_0$. Also, the frequency scales of the relative Gramians change as $L$ and $\sigma$ scale with the respective distance to the origin. We leave this as an area for improvement in our analysis.  
\end{rem}

\section{Proof of Theorems \ref{maintheorem} and \ref{nospilltheorem}}\label{proof}
We re-label where it is understood so that we are now studying the randomized initial field, e.g. $c_n\mapsto \beta^{\nu}_nc_n$. We drop the subscript $x_0$ where it is understood.  The estimates on the randomized field are for arbitrary bounded times $T<\sigma^2$, and we only need the short time assumption for the approximation to randomized observability constant.

{From the previous section, we conclude that our appoximate randomized obervability constant with heat packets can replace the observability constant with $u_{\nu}$ a randomized field as in the Introduction.  However, we cannot obtain an appoximation of the true observability constant because the off diagonal terms contribute substantially to the matrices required, by Proposition \ref{omegabound} c,d). In a deterministic problem, because the heat packets are not orthogonal, one obtains a generalized eigenvalue problem $G c = \lambda H c$, which only becomes an honest eigenvalue problem if one discards that $H$ is not diagonal, i.e.~if one neglects the $L^2$-inner products between different $\phi_n$ and assumes they \emph{are} orthogonal: $H_{nm}= \int_\Omega \phi_n(T,x) \phi_m(T,x)\,dx\sim C_n(T)\delta_{nm}$.  All the analysis below is only possible if $H_{nm}=\int_\Omega \phi_n(T,x) \phi_m(T,x)\,dx\sim C_n(T)\delta_{nm}$, whence the need for randomization. Without this assumption, one is forced to deal with $G c = \lambda H c$, and an analysis of both $G$ and $H$ is required: The relevant matrix is then $G$ in the basis eigenvectors of $H$, and one needs that this is dominated by the diagonal. In the previous section, we were able to reduce the question to a \emph{finite} matrix optimiztation with a lower bound on the observability constant. One could obtain an poor upper bound in terms of only the diagonal entries, but this bound is not very sharp.}

 We need the following proposition to aid in the computations. Let $\mathcal{U}_M= \{\chi_\omega : \omega \in \mathcal{M}_M\}$ the set of characteristic functions supported on sets in $ \mathcal{M}_M$. We study the optimal observability constant as a functional on $\mathcal{U}_L$: 
\begin{equation} \label{CTchi}
C_T(\chi_\omega) = \inf \frac{\int_0^T \int_\Omega \chi_\omega(x) |u(t,x)|^2 dx\ dt}{\int_\Omega |u(T,x)|^2 dx}\ ,
\end{equation}
where the infimum extends over all solutions $u \in C^\infty(\Omega)$ of the initial--boundary problem for the heat equation with homogeneous Dirichlet boundary conditions and initial condition $u(0, \cdot) \in C_c^\infty(\Omega)$.

To assure existence of minimizers, it will be useful to study a relaxed problem, in which we extend $C_T(\chi_\omega)$ from $\mathcal{U}_M$ to its  closure in $L^\infty$ with respect to the weak$^\ast$ topology, $$\overline{\mathcal{U}}_M=\{a\in L^{\infty}(\Omega; [0,1]) : \int\limits_{\Omega}a(x)\,dx=M|\Omega|\}\ .$$ We set: 
\begin{equation} \label{CTa}
C_T(a) = \inf \frac{\int_0^T \int_\Omega a(x) |u(t,x)|^2 dx\ dt}{\int_\Omega |u(T,x)|^2 dx}\ .
\end{equation}

In order to understand the existence and properties of $\chi_\omega$ which maximize $C_T$, we would like to replace $u(0,\cdot)$ by a superposition $\sum_{n} c_n \phi_{n}$ of heat packets as in Lemma \ref{kvlemma}. Then we define $C_T(\chi_\omega)$ as the infimum over all admissible choices of $c_n$. 

From the discussion above we focus on the analysis in the upper half plane.
 As in \cite{ptzheat}, we may write this as an eigenvalue problem, i.e.~with constraint \begin{align*}
\sum_{n} |c_n|^2 = 1.
\end{align*}

 However, here 
 the time--dependence is not entirely trivial.  From Corollary \ref{majorization}a) and b) we may then conclude that $C_T(\chi_\omega)>0$, i.e.~part a) of Theorem \ref{maintheorem}.
Our proof of Theorem \ref{maintheorem} and Theorem 1.2 will now follow closely along the lines of \cite{ptzheat}. The basic hypotheses on the spectral decompositions there can be proven for our heat packets, but we retain the ordering ($H_1$) and ($H_2$) from \cite{ptzheat}. These seemingly natural hypotheses imply strong assumptions on the level sets of the eigenfunctions for the Dirichlet Laplacian which are never actually proved in \cite{ptzheat}. 

The following Lemma will be used to show that minimizers of the relaxed problem are characteristic functions:
\begin{lem}[$H_1$] Assume that there exist a subset $E\subset\Omega$ with $|E|>0$, an integer $N\in \mathbb{N}^*$, coefficients $\alpha_j\in \mathbb{R}_+$, $|j|\leq N$ and $C\geq 0$ such that
\begin{align*}
\sum\limits_{|j|\leq N}\alpha_j\int\limits_0^T|\phi_j(t,x)|^2\,dt=C
\end{align*}
a.e.~on $E$. Then $C=0$ and $\alpha_j=0$ for all $j$.
\end{lem}
\begin{proof} The functions 
\begin{align*}
|\phi_{j}(t,x)|^2 = \left(\frac{\sigma}{\sqrt{2 \pi}(\sigma^2 + t)}\right)^{d} \exp\left(- \frac{|x-x_0|^2-4 |\xi_j|^2 t^2}{2 (\sigma^2+t)} \right)\exp\left(-2|\xi_j|^2t\right)
\end{align*}
extend from $\Rd$ to holomorphic functions of $x$ on $\mathbb{C}^d$. Integrating in $t$, also $\sum\limits_{|j|\leq N}\alpha_j\int\limits_0^T|\phi_j(t,x)|^2\,dt$ admits a holomorphic extension and is constant on $E$. Because $|E|>0$, $E$ contains an accumulation point in $\overline{\Omega}$, and therefore 
\begin{align*}
\sum\limits_{|j|\leq N}\alpha_j\int\limits_0^T|\phi_j(t,x)|^2\,dt = C
\end{align*}
is constant for all $x \in \mathbb{C}^d$. We integrate both sides of this identity over $x \in \mathbb{R}^d$ to conclude that 
\begin{align*}
\sum\limits_{|j|\leq N}\alpha_j \int\limits_0^T\left(\frac{\sigma^2}{\sigma^2 + t}\right)^{d/2}\exp\left(-\frac{2 t\sigma^2|\xi_j|^2}{\sigma^2 + t}\right)\ dt
\end{align*}
is infinite, whenever $C>0$, a contradiction. Therefore $C=0$, and therefore 
\begin{align*}
\sum\limits_{|j|\leq N}\alpha_j \int\limits_0^T \left(\frac{\sigma^2}{\sigma^2 + t}\right)^{d/2}\exp\left(-\frac{2t\sigma^2|\xi_j|^2}{\sigma^2 + t}\right)\ dt = 0.
\end{align*}
As all summands are nonnegative, and the $t$--integrals positive, we conclude $\alpha_j=0$ for all $j$. 
\end{proof}
Let 
\begin{align*}
d_j=\left(\int\limits_{\Omega}|\phi_j(x,T)|^2\,dx\right)^{-1}\ .
\end{align*}
In order to assure the existence of a solution, we use the relaxation as defined in \cite{bu}. Because the set $\mathcal{U}_M$ is not weak-* compact, we consider the convex closure of $\mathcal{U}_M$ in the weak-* topology of $L^{\infty}$, which is then
\begin{align*}
\overline{\mathcal{U}}_M=\{ a\in L^{\infty}(\Omega;[0,1]) | \int\limits_{\Omega}a(x)\,dx=M|\Omega| \}.
\end{align*} 
This relaxation was used in \cite{mu, ptz, opt}. If we replace $\chi_{\omega}\in\overline{U}_M$ with $a\in \overline{U}_M$, we define a relaxed formulation of the optimal shape design problem, by 
\begin{align*}
\sup\limits_{a\in\overline{U}_M} J(a),
\end{align*}
where the functional $J$ naturally extends to $\overline{U}_M$ as
\begin{align*}
J(a)=\inf\limits_{j\in\mathbb{N}}d_j\int\limits_{\Omega}a(x)|\phi_j(x)|^2\,dx.
\end{align*}
We show the following:
\begin{lem}
[$H_2$] For all $a\in\overline{\mathcal{U}}_M$ 
one has
\begin{align}\label{gamma}
\liminf_{|n|\rightarrow\infty}\int_0^T\int_\Omega a(x)d_n|\phi_n(t,x)|^2\,dx\,dt>\gamma_1(T) \ .
\end{align}
Here $\gamma_1(T)$ is the ``appropriately renormalized,'' first--frequency heat packet
\begin{align*}
\gamma_1(T)=\frac{\int_0^T \int_\omega|\phi_1(t,x)|^2\,dx\,dt}{\int\limits_{\Omega}|\phi_1(T,x)|^2\,dx} \ .
\end{align*}
\end{lem}
\begin{proof}
It suffices to prove the estimate for $\chi_{\omega}\in \mathcal{U}_M$,  because $a$ is in the weak-* closure of $\mathcal{U}_M$. The estimate (\ref{gamma}) is equivalent to
\begin{align*}
\frac{\int_0^T \int_\omega|\phi_1(t,x)|^2\,dx\,dt}{\int_{\Omega}|\phi_1(T,x)|^2\,dx}< \lim_{|n|\rightarrow\infty}\frac{\int_0^T \int_\omega|\phi_n(t,x)|^2\,dx\,dt}{\int\limits_{\Omega}|\phi_{n}(T,x)|^2\,dx}
\end{align*}
We use the shorthand $B_t$ for the ball defined by \eqref{truncation}. Let $m_t^{\omega}=\frac{|\omega\cap B_t|}{|B_t|}$ and $C_B=\exp(-1)(\mathrm{erf}(1))$ as in Proposition \ref{omegabound} a,b). The left hand side has the upper bound
\begin{align}\label{u1}
(m_t^{\Omega})^{-1}C_d\frac{(\sup\limits_t m_t)\int\limits_0^TA(t,1)\,dt}{C_BA(T,1)}<(m_t^{\Omega})^{-1}C_dC_B^{-1}|\omega|\frac{\exp(2|\xi_1|^2T)-1}{2|\xi_1|^2}<\infty
\end{align}
and the right hand side has the lower bound
\begin{align}\label{u2}
(m_t^{\Omega})^{-1}C_dC_B\frac{(\inf\limits_t m_t)\int\limits_0^TA(t,n)\,dt}{A(T,n)}>(m_t^{\Omega})^{-1}C_dC_B\frac{|\omega|}{|B_T|}\frac{\exp(2|\xi_n|^2T)-\exp(3/2|\xi_n|^2T)}{3/2|\xi_n|^2}
\end{align}
where we used the fact 
\begin{align*}
\sup_t m^{\omega}_t=|\omega|\qquad \inf_t m^{\omega}_t=\frac{|\omega|}{|B_t|}
\end{align*}
The right hand side of \eqref{u2} goes to $\infty$ as $n\rightarrow \infty$ and the claim is verified. This computation can also be used to show $C_T^A(\omega)>0$ directly, provided we recall the denominator in the randomized observability constant is bounded below by $\sum\limits_n|c_n|^2$ by energy estimates in the Appendix.
\end{proof}

Standard variational arguments assure the existence of a unique relaxed solution:
\begin{lem}\label{existence}
The optimal design problem admits a unique solution $a^*\in \overline{\mathcal{U}}_M$.  
\end{lem}
\begin{proof}
As a result of Corollary \ref{majorization} it suffices to examine the minimization problem for the functional $\mathcal{J}$ on $\overline{\mathcal{U}}_M$, 
\begin{align*}
 \mathcal{J}(a) = \inf_{c \in \ell^2(\mathbb{N}^d)} \frac{\int_0^T \int_\Omega a(x) |\sum\limits_nc_n\phi_n(t,x)|^2 dx\ dt}{\int_\Omega \sum\limits_n| \ c_n\phi_n(T,x)|^2 dx}.\
\end{align*}
 If we consider the normalization $\int_\Omega \sum\limits_n| \ c_n\phi_n(T,x)|^2 dx=1$, then it follows that we may equivalently consider
\begin{align*}
\inf \int_0^T \int_\Omega a(x) \sum\limits_nd_n|\phi_n(t,x)|^2 dx\ dt\ .
\end{align*}
For  $a\in \overline{\mathcal{U}_M}$, the mapping
\begin{align*}
a\mapsto d_n\int_0^T\int\limits_{\Omega}a(x)|\phi_n(t,x)|^2\,dx\,dt
\end{align*}
is linear and continuous in the weak--$*$ topology of $L^{\infty}$. Therefore $J$ is upper semicontinous as the infimum of continuous linear functionals. Because $\overline{\mathcal{U}_M}$ is compact in the weak--$*$ topology, the result follows. 
\end{proof} 

\begin{proof}[Proof of Theorem 1.1 b), c)]
We now define the truncated functional
\begin{align*}
J_N(a)=\inf\limits_{|n|\leq N}d_n\int_0^T\int_\Omega a(x)|\phi_n(t,x)|^2\,dx\,dt
\end{align*}
for all $a\in \overline{\mathcal{U}}_M$ and consider the problem
\begin{align*}
\sup\limits_{\overline{\mathcal{U}}_L}J_N(a) \ .
\end{align*}
The above problem has at least one solution $a^N\in\overline{\mathcal{U}}_L$, by the arguments that proved Lemma \ref{existence}, because $\overline{\mathcal{U}}_L$ is weak--$*$ compact and $J_N$ upper semi--continuous. Using $(H_1)$ and $(H_2)$, we are going to  show that the solution $a^N$ is actually a characteristic function of a set $\omega^N$, $a^N=\chi_{\omega^N}\in \mathcal{U}_M$. Namely, if we denote
\begin{align*}
\mathcal{S}_N=\{\alpha=(\alpha_n)_{ |n|\leq N} : \sum\limits_{|n|\leq N} \alpha_n=1\}\ ,
\end{align*}
the Sion minimax theorem implies
\begin{align*}
&\sup\limits_{a\in\overline{\mathcal{U}}_M}\min\limits_{1\leq n\leq N}d_n\int_0^T\int_\Omega a(x)|\phi_n(t,x)|^2\,dx\,dt\\&=
\max\limits_{a\in \overline{\mathcal{U}}_M}\min\limits_{\alpha\in \mathcal{S}_N}\int_{\Omega}a(x) \varphi_N(x)\,dx\\& =
\min_{\alpha\in\mathcal{S}_N}\max\limits_{a\in\overline{\mathcal{U}}_M}\int_{\Omega}a(x) \varphi_N(x)\,dx\ .
\end{align*}
Here we have defined 
\begin{align*}
\varphi_N(x)=\sum\limits_{|n|\leq N}d_n\alpha_n\int\limits_{0}^T|\phi_n(t,x)|^2\,dt\ .
\end{align*}
As a result, there exists $\alpha^N\in \mathcal{S}_N$ such that $(a^N,\alpha^N)$ is a saddle point of the functional. 
This then implies that $a^N$ is a solution to the problem
\begin{align*}
\max_{a\in\overline{\mathcal{U}}_M}\int\limits_{\Omega}a(x)\varphi_N(x)\,dx \ .
\end{align*}
By $(H_1)$, the functional $\varphi_N$ cannot be constant on a subset of $\Omega$ of positive measure. This implies the existence of $\lambda^N$ such that $a^N(x)=1$ when $\varphi_N(x)>\lambda_N$ and $a^N(x)=0$ otherwise. As $J_N$ is concave, the set of maximizers is convex. Since every maximizer is  a characteristic function, $a^N \in \mathcal{U}_M$ must be the unique maximizer. We note that $\varphi_N$ is analytic and therefore $\omega_N$ is open and semi-analytic.
\end{proof}

\begin{proof}[Proof of Theorem 1.2]
It remains to compare the maximizers $a^N \in \mathcal{U}_M$ of $J_N$ with the maximizer $a^*\in \overline{\mathcal{U}}_M$ of $J$ from Lemma \ref{existence}. First,
\begin{align*}
J_{N_0}(a^*)\leq \gamma_1(T)\ ,
\end{align*}
and $(H_2)$ applied to $a^*$ shows that 
\begin{align}\label{given}
\inf\limits_{|n|>N_0}d_n\int\limits_{\Omega}a^*(x)\int\limits_0^T|\phi_n(t,x)|^2\,dx\,dt > \gamma_1(T) 
\end{align}
for some $N_0\in \mathbb{N}$.
From \eqref{given} we have that
\begin{align*}
J(a^*)=\min\left\{ J_{N_0}(a^*), \inf\limits_{|n|>N_0} d_j\int_0^T\int_\Omega  a^*(x)|\phi_n(t,x)|^2\,dx\,dt\right\}=J_{N_0}(a^*)\ .
\end{align*}
If $a^{N_0}\in \mathcal{U}_M$ is the maximizer of $J_{N_0}$, we show that $J(a^*)=J_{N_0}(a^{N_0})$ as in \cite{ptzheat}: Indeed, because $a^{N_0}$ maximizes $J_{N_0}$ over $\overline{\mathcal{U}}_M$, we have that $J(a^*)=J_{N_0}(a^*)\leq J_{N_0}(a^{N_0})$. We assume 
\begin{align}\label{contradict}
J_{N_0}(a^*)<J_{N_0}(a^{N_0})
\end{align}
and obtain a contradiction: As $J_{N_0}$ is concave, for every $t \in (0,1]$ the assumption \eqref{contradict} implies
$$J_{N_0}(a^* + t(a^{N_0}-a^*)) \geq (1-t)J_{N_0}(a^*) + t J_{N_0}(a^{N_0}) > J_{N_0}(a^*) = J(a^*)\ .$$ By the choice of $N_0$, one concludes that $$J_{N_0}(a^* + t(a^{N_0}-a^*)) = J(a^* + t(a^{N_0}-a^*))  >  J(a^*)\ ,$$ 
in contradiction to $a^*$ being a maximizer of $J$. So indeed, $J_{N_0}(a^*)=J(a^*) =J_{N_0}(a^{N_0})$, or $a^* = a^{N_0}$.
\end{proof}

\begin{rem} Even though we are finding our optimal set over all indices $n$, we remark that the truncation of the admissible indices in Lemma \ref{kvlemma} implies a bound on $N$ For the case of high frequency data where $\epsilon_0\rightarrow 0$, $N$ increases enormously. 
\end{rem}

\section{Appendix: Well-posedness Estimates for the Heat Equation} 
We consider the problem:
\begin{align}\label{ibvp1}
\partial_tu&=\Delta u \nonumber\\ 
u(t,x)|_{x\in\partial\Omega}&\equiv 0 \\\nonumber
u(0,x)&=g(x)
\end{align}
with $g(x)\in C_c^{k}(\Omega)$. We claim:
\begin{thm}\label{wp}
Let $k \in \mathbb{N}_0$. Problem \eqref{ibvp1} admits the following well-posedness estimate:
\begin{align}\label{est}
\|u\|_{L^{\infty}(0,T);H^k(\Omega))}+\|u\|_{L^2(0,T);H^{k+1}(\Omega))}\leq C\|g\|_{H^k(\Omega)}\ .
\end{align} 
\end{thm}
\begin{proof}
We multiply \eqref{ibvp1} by $u$ and integrate over $\Omega$ while applying the divergence theorem. We note that the boundary condition gives us:
\begin{align*}
\frac{1}{2}\frac{d}{dt}\int\limits_{\Omega}u^2(t,x)\,dx+\int\limits_{\Omega}|Du(t,x)|^2\,dx=0\ .
\end{align*}
Integrating this equation with respect to time and using the initial condition we obtain
\begin{align*}
\frac{1}{2}\int\limits_{\Omega}u^2(t,x)\,dx+\int\limits_0^t\int\limits_{\Omega}|Du(s,x)|^2\,dx\,ds=\frac{1}{2}\int\limits_{\Omega}g^2(x)\,dx\ .
\end{align*}
Taking the supremum over $T$ gives the desired result for $k=0$. The result for $k>0$ follows by differentiating the heat equation and choosing $\nabla^ku$ as a test function. 
\end{proof}
  
\section*{Acknowledgments}

A.~W.~acknowledges support by EPSRC grant EP/L01937X/1 and ERC Advanced Grant MULTIMOD 26718. H.~G.~is supported by a PECRE award of the Scottish Funding Council and ERC Advanced Grant HARG 268105. Both authors would like to thank James Ralston for noticing some mistakes in the earlier version.

\end{document}